\documentclass[preprint,11pt]{elsarticle}

\usepackage{amssymb}
 \usepackage{amsthm}
 \usepackage{amsmath}
 \usepackage{graphicx}
 \usepackage{subfigure}
 \usepackage{lscape}
\newtheorem{theorem}{Theorem}[section]

\newtheorem{lemma}[theorem]{Lemma}

\theoremstyle{definition}

\newtheorem{remark}{Remark}

\journal{ }

\begin{document}

\begin{frontmatter}



\title{Splitting properties of linear differential systems with small delays}

\author{Shuang Chen}
\ead{schen@hust.edu.cn}
\address{Center for Mathematical Sciences, Huazhong University of Sciences and Technology\\
Wuhan, Hubei 430074, P. R. China}

\begin{abstract}

We study the splitting properties of general linear differential equations with small delay.
After giving the explicit expressions of the bounds and the exponents
associated with a certain dichotomy induced by small delay,
we prove that as the delay tends to zero,
the spectral gap approaches to infinity,
and  the angular distance and the separation index associated with this dichotomy are bounded from below by a positive constant
which is independent of the delay.

\end{abstract}


\end{frontmatter}



\section{Introduction}
\label{sec-intr}

The effects of small delays on the dynamics of differential equations have been widely studied in the past decades.
For example,
spectral analysis and special solutions for linear differential equations with small delays were investigated in
\cite{Arino-Pituk,Chen-Shen-20,Driver1976,Faria-Huang,Gyori-Pitukl}.
Asymptotic behaviors and inertial manifolds for nonlinear equations were considered in
\cite{Chicone03,Driver1968,Ryabov1965}.
Existence of periodic solutions for scalar differential equations with small delays
were studied in \cite{Chow-Huang94,Chow-Mallet83}.
Recently,
traveling waves arising from partial differential equations with small  nonlocal delays \cite{DuLiLi-18,Li-Guo-18},
and canard explosion and relaxation oscillations in differential systems with small delays \cite{Campbel-Stone-Erneux,Krupa-Touboul}
also attracted many attentions.
For many more dynamical behaviors induced by small delays,
we refer to  \cite{Kuehn-15,Walther-14}.
Here we consider a general linear non-autonomous retarded differential equation with small delay in the view of exponential dichotomy
\cite{Barreira-D-Valls,Coppel},
and study the effects of  small delay on the spectral gap, the angular distance and the separation index
\cite{Daleckii-Krein-74,Lian-Wang-15,Massera-66}.

Exponential dichotomy is an important concept in describing the hyperbolic property of linear differential equations.
It is widely used in studying invariant manifolds, homoclinic bifurcation,  linearization and so on \cite{Barreira-D-Valls,Coppel,JKHale-Verduyn,Henry1981,Zhang-etal-17}.
As a generalization of exponential dichotomy,
the concept of pseudo-exponential dichotomy was laid to establish the weak stable/unstable invariant manifolds
(see, for instance, \cite{Hirsch-Pugh-Shub,Zhang}).
Let $\mathcal{B}$ denote a Banach space.
Assume that $\{T(t,s):t\geq s\}$, a two-parameter family of bounded linear operators on $\mathcal{B}$,
is a semigroup and strongly continuous in $s$ and $t$.
Then the two-parameter family $\{T(t,s):t\geq s\}$  is said to admit a {\it pseudo-exponential dichotomy} on an interval $J(\subset \mathbb{R})$
if for each $s\in J$,
there exist a projection $P(s)$ on $\mathcal{B}$ and real constants $K$, $\alpha$ and $\beta$ with $K>0$ and $\beta<\alpha$
such that the following properties hold:
\begin{itemize}
\item[(i)]
$T(t,s)P(s) = P(t)T(t,s)$ for $t\geq s$ in $J$.
\item[(ii)]
$T(t,s)|_{\mathcal{R}(P(s))}$ is an isomorphism from $\mathcal{R}(P(s))$ onto $\mathcal{R}(P(t))$,
where $\mathcal{R}(P(s))$ is the range of $P(s)$.
The inverse of $T(t,s)|_{\mathcal{R}(P(s))}$ is denoted by $T(s,t):\mathcal{R}(P(t)) \to \mathcal{R}(P(s))$.
\item[(iii)]
$|T(s,t)P(t)\phi| \leq K e^{-\alpha (t-s)}|P(t)\phi|$ for $t\geq s$ in $J$ and $\phi\in\mathcal{B}$.
\item[(iv)]
$|T(t,s)Q(s)\phi| \leq K e^{\beta (t-s)}|Q(s)\phi|$ for $t\geq s$ in $J$ and $\phi\in\mathcal{B}$,
where $Q(s)=I-P(s)$ and $I$ is the identity.
\end{itemize}
We refer  to the constants $K$, $\alpha$ and $\beta$ as a {\it bound}, an {\it upper exponent} and a {\it lower exponent}
of this pseudo-exponential dichotomy, respectively.
The {\it spectral gap} associated with this dichotomy is defined by:
\begin{eqnarray*}
\mbox{ the spectral gap }\!\!\!&=&\sup\left\{\alpha\in\mathbb{R}:\sup_{t\geq s}|T(s,t)P(t)| e^{\alpha (t-s)}<+\infty\right\}\\
& &-\inf\left\{\beta\in\mathbb{R}:\sup_{t\geq s}|T(t,s)Q(s)|e^{-\beta (t-s)}<+\infty\right\}.
\end{eqnarray*}
Whenever there is no confusion,
we always use $|\cdot|$ to denote the norms of the elements in a linear space endowed with a suitable norm.
The spectral gap plays an important role in the invariant manifold reduction and
its value guarantees the smoothness of invariant manifolds.
We refer to  \cite{Bates-Jones,Chow-Lu,Foiaseatal,Henry1981,Hirsch-Pugh-Shub,Shen-Lu-Zhang-20} for more details on
the relation between the spectral gap and the invariant manifolds theory.

Assume that a strongly continuous semigroup  $\{T(t,s):t\geq s\}$ admits a pseudo-exponential dichotomy.
To describe the splitting properties of the dichotomy,
we give the definitions of the {\it angular distance}  and the {\it separation index} between the nonzero spaces $X_{+}=\mathcal{R}(P(s))$
and $X_{-}=\mathcal{R}(I-P(s))$ for each $s\in J$.
More information on the angular distance and the separation index can be found in \cite{Daleckii-Krein-74,Lian-Wang-15,Massera-66}.
Before giving the definition for the angular distance between two nonzero subspaces,
we first define the angular distance between two nonzero elements in the Banach space $X$ (see \cite{Daleckii-Krein-74,Massera-66}).
For each pair of nonzero elements $\xi_{\pm}$  in $X$,
the angular distance $\gamma(\xi_{+},\xi_{-})$ between $\xi_{+}$ and $\xi_{-}$ is defined by
$$\gamma(\xi_{+},\xi_{-}):=|\frac{\xi_{+}}{|\xi_{+}|}-\frac{\xi_{-}}{|\xi_{-}|}|.$$
Then the angular distance $\gamma(X_{+},X_{-})$ between nonzero subspaces $X_+$ and $X_-$ is given by
\begin{eqnarray}
\gamma(X_+,X_-):=\inf\left\{\gamma(\xi_+,\xi_-): \xi_+\in X_+/\{0\},\ \ \xi_-\in X_-/\{0\} \right\}.
\label{df-distance}
\end{eqnarray}
The angular distance $\gamma(X_{+},X_{-})$ between subspaces $X_+$ and $X_-$ is closely related to
the separation index $\underline{\rm dist}(X_{+},X_{-})$ (see \cite{Lian-Wang-15}),
which is given by
\begin{eqnarray}
\underline{\rm dist}(X_{+},X_{-}):=\inf_{\xi_{+}\in X_{+}\cap S}\left\{\inf_{\xi_{-}\in X_{-}}|\xi_{+}-\xi_{-}|\right\},
\label{df-index}
\end{eqnarray}
where the set $S=\{\xi\in X: |\xi|=1\}$ denotes the unit ball in the Banach space $X$.
Both of them describe the geometric properties of the dichotomies for linear differential systems,
and play important roles in the study of dynamical behaviors.

In this paper,
we investigate the pseudo-exponential dichotomy for a linear retarded differential equation with small delay.
More precisely,
we consider a general linear non-autonomous retarded differential equation of the form
\begin{eqnarray}
\dot x(t) = L(t,x_t),
\label{NARFDE}
\end{eqnarray}
where the linear operator $L(t,\cdot):  C[-r,0]:=C([-r,0],\mathbb{R}^n) \to \mathbb{R}^n$ is continuous for each $t\in \mathbb{R}$,
the constant $r$ is the delay
and the section $x_{t}$ is defined by $x_t(\theta):= x(t+\theta)$ for $\theta \in [-r,0]$.
Let the space $C[-r,0]$ of all continuous functions from $[-r,0]$ into $\mathbb{R}^n$ be equipped with the supremum norm.
Then by the {\it Riesz Representation Theorem},
the operators $L(t,\cdot)$ are represented by
\begin{eqnarray}\label{Lt}
L(t,\phi)=\int^{0}_{-r} d[\eta(t,\theta)]\phi(\theta) ,\ \ \ \
         t\in \mathbb{R}, \ \ \phi \in  C[-r,0] ,
\end{eqnarray}
where the kernel $\eta$ is an $n\times n$ matrix-valued function on $\mathbb{R}\times \mathbb{R}$,
measurable in $(t,\theta)\in \mathbb{R}\times \mathbb{R}$,
and normalized so that $\eta$ satisfies $\eta(t,\theta)=\eta(t,-r)$ for $\theta\leq -r$
and $\eta(t,\theta)=0$ for $\theta\geq 0$,
and $\eta(t,\cdot)$ is continuous from the left on $(-r,0)$ and has bounded variation for each $t\in \mathbb{R}$.
We follow \cite{Arino-Pituk,Chicone03,Driver1976,Ryabov1965} and
assume that the delay $r$ and the kernel $\eta$ satisfy the following hypothesis:
\begin{enumerate}
\item[{\bf (H)}]
there is a positive constant $M$ such that
the total variation ${\rm Var}_{[-r,0]}\eta(t,\cdot)$ of $\eta(t,\cdot)$ on $[-r,0]$ and the delay $r$ respectively satisfy
\begin{eqnarray*}
{\rm Var}_{[-r,0]}\eta(t,\cdot) \leq M \mbox{ for each } t\in\mathbb{R}\ \mbox{ and }\
0<r<r_{0}:=1/(Me).
\end{eqnarray*}
\end{enumerate}
By \cite[Theorem 2.3, p.44]{JKHale-Verduyn},
linear equation (\ref{NARFDE}) with the initial value $x_s=\phi$ has a unique solution $x(\cdot,s,\phi)$.
Define the solution operator $T(t,s): C[-r,0]\to  C[-r,0] $ of equation (\ref{NARFDE}) by
\begin{eqnarray}
T(t,s)\phi:= x_t(\cdot,s,\phi) \ \mbox{ for each } s,t\in \mathbb{R} \mbox{ with } t\geq s \mbox{ and } \phi\in  C[-r,0] .
\label{df-semigroup}
\end{eqnarray}
Then the two-parameter family $\{T(t,s):t\geq s\}$ of the solution operators acting on the space $ C[-r,0] $
is an evolutionary system and strongly continuous in $t$ and $s$ (see \cite{JKHale-Verduyn}).

Our goal is to study the splitting properties of linear differential systems with small delays.
In the current paper,
we prove that the spectral gap corresponding to  a pseudo-exponential dichotomy for equation (\ref{NARFDE}) with {\bf (H)}
tends to infinity as the delay approaches to zero (see Theorem \ref{thm-main-results}),
and obtain that both the corresponding angular distance and  separation index
are uniformly bounded from below for sufficiently small delay (see Theorem \ref{thm-angular-index}).
These results improve the existing results in the literatures \cite{Arino-Pituk,Driver1968,Driver1976}.
By the invariant manifolds theory (see, for instance, \cite{Bates-Jones,Chow-Lu,JKHale-Verduyn,Henry1981}),
we find that large spectral gap is helpful to realize the finite-dimensional reduction of  retarded differential equations.
As a consequence,
we could apply the methods for ordinary differential equations to study local dynamics of small-delay systems,
such as Hopf bifurcation and canard explosion
(see \cite{Campbel-Stone-Erneux,Erneux-etal,Krupa-Touboul,Kuehn-15} and the references therein).

\section{Related work and main results}
\label{sec-2}

In this section,
we start by introducing some related work as preliminaries,
and then  state the main results in this paper.

Under the hypothesis {\bf (H)},
\cite{Driver1976} established the existence of the so-called {\it special matrix solution} \cite{Driver1976}
for linear retarded differential equation (\ref{NARFDE}),
which is similar to the fundamental matrix solution  of a linear nonautonmous ordinary differential equation.
The special matrix solution for equation (\ref{NARFDE}) with the hypothesis {\bf (H)} has the following properties.
\begin{theorem}[Driver, 1976, Theorem 3]
Assume that equation (\ref{NARFDE}) satisfies the hypothesis ${\bf (H)}$.
Then there exists a unique $n\times n$ matrix-valued function $\Phi$,
which is defined on $\mathbb{R}\times \mathbb{R}$
and satisfies that for each $t_{0}\in\mathbb{R}$:
\begin{itemize}
\item[(i)]
 each column of $\Phi(\cdot,t_{0})$ is a solution of equation (\ref{NARFDE}) on $\mathbb{R}$.

\item[(ii)] $\Phi(t_{0},t_{0})=I$, where $I$ is the identity.

\item[(iii)] $|\Phi(t,t_0)|e^{(t-t_0)/r}$ is bounded for each $t\leq t_{0}$.

\item[(iv)]
for each $t\in \mathbb{R}$, the matrix $\Phi(t,t_{0})$ is nonsingular and $\Phi(t,t_0)^{-1}=\Phi(t_0,t)$,
and for each $t_1,t_2\in \mathbb{R}$,
$\Phi(t_2,t_1)\Phi(t_1,t_0)=\Phi(t_2,t_0)$.

\item[(v)] $|\Phi(t,t_0)|\leq e^{\lambda_{r}|t-t_0|}$ for each $t\in \mathbb{R}$,
where  the constant $\lambda_{r}$ is the unique real root of equation $Me^{r\lambda}=\lambda$ in the interval $(-1/r,0)$.
\end{itemize}
\label{thm-prty-sp-solu}
\end{theorem}

Following the results stated in  \cite{Driver1968,Driver1976},
\cite{Arino-Pituk} proved the existence of a pseudo-exponential dichotomy
and obtained the explicit expression of the corresponding projection
by applying the formal adjoint equations associated with linear retarded differential equations
(see \cite[Chapter 6]{JKHale-Verduyn} and \cite{Henry1971}).
As shown in \cite[Section 4]{Arino-Pituk},
the existence of a pseudo-exponential dichotomy for equation (\ref{NARFDE}) with  the hypothesis ${\bf (H)}$
is summarized in the following theorem.

\begin{theorem}[Arino \& Pituk, 2001]
Assume that equation (\ref{NARFDE}) satisfies the hypothesis ${\bf (H)}$.
Then $\{T(t,s):t\geq s\}$ defined by (\ref{df-semigroup}) admits a pseudo-exponential dichotomy.
More precisely,
there exist projections $P(s)$, $s\in\mathbb{R}$,
and  constants $K_{1,r}$, $K_{2,r}$, $\alpha_{r}$ and $\beta_{r}$ with $\alpha_{r}>\beta_{r}$ such that
for each $s,t \in\mathbb{R}$ with $t\geq s$,
the following statements hold:
\begin{itemize}
\item[(i)]
$T(t,s)P(s) = P(t)T(t,s)$.
\item[(ii)]
$T(t,s)|_{\mathcal{R}(P(s))}$ is an isomorphism from $\mathcal{R}(P(s))$ onto $\mathcal{R}(P(t))$,
where $\mathcal{R}(P(s))$ is the range of $P(s)$
and the dimension ${\rm dim}\mathcal{R}(P(s))$ of $\mathcal{R}(P(s))$ satisfies
${\rm dim}\mathcal{R}(P(s))=n$ for each $s\in\mathbb{R}$.
The inverse of $T(t,s)|_{\mathcal{R}(P(s))}$ is denoted by $T(s,t):\mathcal{R}(P(t)) \to \mathcal{R}(P(s))$.
\item[(iii)] for  each $\phi\in  C[-r,0] $,
\begin{eqnarray}
|T(s,t)P(t)\phi| &\leq & K_{1,r} e^{\alpha_{r}(s-t)}|P(t)\phi|,\label{ged-na-est-1}\\
|T(t,s)(I-P(s))\phi| &\leq &
 K_{2,r}e^{\beta_{r}(t-s)}|(I-P(s))\phi|.\label{ged-na-est-2}
\end{eqnarray}
\end{itemize}
\label{thm-ed}
\end{theorem}
We remark that the representations of
the projections $\{P(s):s\in\mathbb{R}\}$ are too complicated,
thus they are  omitted here and the readers are referred to \cite[p.403]{Arino-Pituk}.

As stated in Theorem \ref{thm-ed},
the existence of a pseudo-exponential dichotomy for equation (\ref{NARFDE}) with the hypothesis {\bf (H)}
is obtained.
However, it does not give the explicit expressions of the bounds and the exponents for this dichotomy.
We provide these expressions in the paper and further prove that
the corresponding spectral gap approaches to infinity as the delay $r$ tends to zero.
More precisely, we have the following statements.
\begin{theorem}
Assume that equation (\ref{NARFDE}) satisfies the hypothesis ${\bf (H)}$.
Then  for each $s\in \mathbb{R}$,
the following assertions hold:
\begin{itemize}
\item[(i)]
the constants $\alpha_{r}$, $\beta_{r}$ and $K_{i,r}$ in (\ref{ged-na-est-1}) and (\ref{ged-na-est-2})
can be as follows:
\begin{eqnarray}
\begin{aligned}
&\alpha_{r}=- \lambda_{r} ,\ \ \
\beta_{r}=\frac{\rho\ln(r \lambda_{r})}{r}- \lambda_{r},\\
&K_{1,r}=1, \ \ \
K_{2,r}=-\frac{2e^{2+r \lambda_{r}}(r \lambda_{r})^{1-2\rho}}{\rho\ln(r \lambda_{r})},
\end{aligned}
\label{constants}
\end{eqnarray}
for every $\rho\in (0,1]$,
where the constant $\lambda_{r}$ is defined as in Theorem \ref{thm-prty-sp-solu}.
\item[(ii)]
the norms $|P(s)|$ and $|I-P(s)|$ of the projections $P(s)$ and $I-P(s)$ have
the following uniform bounds with respect with $s$:
\begin{eqnarray}
\max\{|P(s)|, |I-P(s)|\}\leq 2e\gamma_{r}(K_{r})^{2\gamma_{r}-1},
\label{est-prj-01}
\end{eqnarray}
where the constants $\gamma_{r}$ and $K_{r}$ are respectively given by
$$\gamma_{r}=(M-\beta_{r})/(\alpha_{r}-\beta_{r}), \ \ \ K_{r}=\max\{K_{1,r},\,K_{2,r}\}.$$
\end{itemize}

Furthermore,
consider a family of linear retarded differential equations of the form (\ref{NARFDE}),
which satisfy the hypothesis ${\bf (H)}$ and are parameterized by the delay $r$.
Then we have the following:
\begin{itemize}
\item[${\rm(iii)}$] as the delay $r\to 0+$,  the spectral gap $\alpha_{r}-\beta_{r}\to +\infty$
and the constant
\begin{eqnarray}
\mathcal{L}_{r}:=\frac{\alpha_{r}-\beta_{r}}{4(K_{r})^{2}\max\{|P(s)|, |I-P(s)|\}}
     \to +\infty.
\label{gap-cond}
\end{eqnarray}
\end{itemize}
\label{thm-main-results}
\end{theorem}

To show that the limit in (\ref{gap-cond}) is useful in the invariant manifold reduction,
we consider the following nonlinear retarded differential equation
\begin{eqnarray}\label{nonlinear-NDE}
\dot x(t)= Lx_{t}+f(x_t),
\end{eqnarray}
for each $t\geq 0$ and each $x\in \mathbb{R}^{n}$,
where $L$ is the linear operator defined by (\ref{Lt}) whose the kernel $\eta$ is independent of $t$,
and the nonlinear term $f:C[-r,0]\to\mathbb{R}^{n}$ satisfies the following hypothesis:
\begin{enumerate}
\item[{\bf (H')}]
the map $f$ is continuous, $f(0)=0$ and globally Lipschitz on $C[-r,0]$, that is, for all $\phi_{1}$ and $\phi_{2}$ in $C[-r,0]$,
there exists a constant $\mathcal{K}_{f}$ such that
\begin{eqnarray*}
|f(\phi_1)-f(\phi_2)|\leq \mathcal{K}_{f}|\phi_1-\phi_2|.
\end{eqnarray*}
\end{enumerate}
To construct the invariant manifolds
in the neighbourhood of a singular point $x=0$ for nonlinear system (\ref{thm-main-results}) with the hypothesis {\bf (H')},
one usually assumes that the so-called {\it spectral gap condition} holds (see, for instance, \cite{Bates-Jones,JKHale-Verduyn,Hirsch-Pugh-Shub}),
that is,
\begin{eqnarray*}
\mathcal{L}_{r}=\frac{\alpha_{r}-\beta_{r}}{4(K_{r})^{2}\max\{|P(s)|, |I-P(s)|\}}>\mathcal{K}_{f}.
\end{eqnarray*}

Theorem \ref{thm-main-results} shows that
the spectral gap condition for retarded differential equations can be easily satisfied if the delay is sufficiently small.
This fact can be used in the study of complex oscillations arising from differential equations with small delays,
such as canard explosion and relaxation oscillation \cite{Campbel-Stone-Erneux,Erneux-etal,Krupa-Touboul}.

Recall that the angular distance and  the separation index associated with a splitting of the Banach space $C[-r,0]$
are defined by (\ref{df-distance}) and (\ref{df-index}).
Based on Theorems \ref{thm-ed} and \ref{thm-main-results},
we further investigate the angular distance and  the separation index for the obtained dichotomy.
\begin{theorem}
Consider a family of linear retarded differential equations of the form (\ref{NARFDE}),
which satisfy the hypothesis ${\bf (H)}$ and are parameterized by the delay $r$.
For each $s\in \mathbb{R}$,
let the subspaces $\mathcal{C}_{\pm}$ be defined by
\begin{eqnarray*}
\mathcal{C}_{+}=\mathcal{R}(P(s)), \ \ \ \ \  \mathcal{C}_{-}=\mathcal{R}(I-P(s)),
\end{eqnarray*}
where $P(s)$ is the projection operator defined as in Theorem \ref{thm-ed}.
Then there exist a sufficiently small $\widetilde{r}_{0}$ with $0<\widetilde{r}_{0}<1/(Me)$
and a positive constant $\delta$ that is independent of $r$ in $(0,\widetilde{r}_{0})$ such that
\begin{eqnarray}
0<\delta\leq \underline{\rm dist}(\mathcal{C}_{+},\mathcal{C}_{-}) \leq \gamma(\mathcal{C}_+,\mathcal{C}_-),\ \ \
0<\delta\leq \underline{\rm dist}(\mathcal{C}_{-},\mathcal{C}_{+}) \leq \gamma(\mathcal{C}_+,\mathcal{C}_-).
\label{est-main-results}
\end{eqnarray}
\label{thm-angular-index}
\end{theorem}

\section{Proof of main results}

In this section we give the detailed proof for main results stated in this paper.
We start by the several properties of an analytic function,
which are useful in proving the last statement in Theorem \ref{thm-main-results}.

\begin{lemma}
Let the function $g: \mathbb{R}\to \mathbb{R}$  be in the form
\begin{eqnarray*}
g(x)=Me^{rx}-x, \quad x\in \mathbb{R},
\end{eqnarray*}
where the parameter $r$ satisfies $0<r<r_0$ and the real constants $M, r_0$ are defined as in the hypothesis ${\bf (H)}$.
Then the function $g$ has precisely two real zeros $\lambda_{r}$ and $\mu_{r}$ satisfying
\begin{eqnarray}\label{xi est}
0<\lambda_{r}<\frac{1}{r_0}<\frac{1}{r}<\mu_{r}.
\end{eqnarray}
Furthermore,
the constant $\lambda_{r}$ satisfies the following limits:
\begin{eqnarray*}
r\lambda_{r}\to 0, \ \ \ \lambda_{r}\to M, \ \mbox{ as }\ r\to 0.
\end{eqnarray*}
\label{lm1.1}
\end{lemma}
\begin{proof}
It is clear that
$g'(x)=Mre^{rx}-1$ and $g''(x)=Mr^2e^{rx}>0$,
then the derivative $g'$ of the function $g$ has a unique real zero
$$\lambda^{0}=-\frac{\ln(Mr)}{r}>-\frac{\ln(e^{-1})}{r}=\frac{1}{r},$$
where the inequality is from the hypothesis {\bf (H)}.
Moreover,
we also obtain that
the function $g$ has at most two real zeros. Since
\begin{eqnarray*}
g(0)=M>0,\quad g(\frac{1}{r})=Me-\frac{1}{r}=\frac{Mer-1}{r}<0,\quad \lim_{x\to+\infty}g(x)=+\infty,
\end{eqnarray*}
then by continuity
there are real constants $\lambda_{r}$ and $\mu_{r}$  satisfying
\begin{eqnarray*}
\lambda_{r}<1/r<\mu_{r}, \ \ \ g(\lambda_{r})=0, \ \ \ g(\mu_{r})=0.
\end{eqnarray*}
Then by {\it Rolle's Theorem}, $$\lambda_{r}<\frac{1}{r}<\lambda^{0}<\mu_{r}.$$ Note that
$$\lambda_{r}=Me^{rx_1(r)}<Me=\frac{1}{r_0}<\frac{1}{r},$$
then we get the inequalities in (\ref{xi est}).

Since $g(\lambda_{r})=0$,
then $rM=r\lambda_{r}e^{-r\lambda_{r}}$.
It implies that
\begin{eqnarray*}
 r\lambda_{r}\to 0 \ \ \ \mbox{ as }\ r\to0,
\end{eqnarray*}
here we use the properties of the function $h(x)=xe^{-x}$ for $x\in \mathbb{R}$.
Then by  $\lambda_{r}=Me^{r\lambda_{r}}$,
we obtain that $\lambda_{r}\to M$ as $r\to0$.
Therefore, the proof is now complete.
\end{proof}

In the next lemma we prove a key inequality, which is useful in subsequent proof.
\begin{lemma}
\label{lm-key-est}
Let $\varphi$ be a continuous function from $\mathbb{R}^+:=[0,+\infty)$ to $\mathbb{R}^+$
and satisfy
\begin{eqnarray}
c_1:=\max_{0\leq t \leq r}\varphi(t) \  \ \mbox{ and } \ \ \varphi(t)\leq c_2 \int_{t-r}^{t} \varphi(s) d s, \ \ \ \ t\geq r,
\label{cond-varphi}
\end{eqnarray}
for a positive constant $c_2$ with $c_2r<1$.
Then the function $\varphi$ satisfies $\varphi(t)\leq c_1$ for each $t\geq 0$ and
\begin{eqnarray*}
\varphi(t) \leq \frac{c_{1}}{(c_{2}r)^{\rho}}\exp\left(\frac{\rho\ln(c_{2}r)}{r}t\right),
\end{eqnarray*}
for each $t\in \mathbb{R}^+$ and each $\rho\in (0,1]$.
\end{lemma}
\begin{proof}
To prove this lemma,
we first claim that $\varphi(t)<c_1$ for $r\leq t\leq 2r$.
Otherwise,
suppose that $\max_{r\leq t \leq 2r}\varphi(t)\geq c_{1}$.
Since
\begin{eqnarray}
\varphi(r)\leq c_2 \int_{0}^{r} \varphi(s) d s\leq c_{1}c_{2}r<c_{1},
\end{eqnarray}
then by continuity,
there exist a constant $t_{1}\in (r,2r]$ such that
\begin{eqnarray*}
t_{1}=\inf\{r\leq t\leq 2r: \varphi(t_{1})=c_{1}\},
\end{eqnarray*}
which implies that $\varphi(t)<c_{1}$ for $r<t<t_{1}$ and $\varphi(t_{1})=c_{1}$.
This is in contradiction with the fact that
\begin{eqnarray*}
\varphi(t_{1})&\leq& c_2 \int_{t_{1}-r}^{t_{1}} \varphi(s) d s\\
              &=& c_2\left(\int_{t_{1}-r}^{r} \varphi(s) d s+\int_{r}^{t_{1}} \varphi(s) d s\right)\\
             &<& c_{2}(c_{1}(2r-t_{1})+c_{1}(t_{1}-r))\\
             &=& c_{2}c_{1}r\\
             &<& c_{1},
\end{eqnarray*}
where we use the following inequalities
\begin{eqnarray*}
\int_{t_{1}-r}^{r} \varphi(s) d s\leq c_{1}(2r-t_{1}), \ \ \ \ \int_{r}^{t_{1}} \varphi(s) d s<c_{1}(r-t_{1}).
\end{eqnarray*}
Hence, the claim holds.
It is not hard to prove that $\varphi(t)\leq c_1$ for all $t\geq 0$ by induction.

Let an auxiliary function $\psi$ from $\mathbb{R}^+$ to itself be defined by
\begin{eqnarray}
\psi(t)=\sup_{s\geq t} \varphi(s) \ \mbox{ for }\ t\geq 0.
\label{df-psi}
\end{eqnarray}
By the above statements,
the function $\psi$ is well-defined.
By  the conditions in (\ref{cond-varphi}),
\begin{eqnarray}
\psi(t)=\sup_{s\geq t}\varphi(s)\leq \sup_{s\geq t}c_2r \psi(s-r)=c_2r \psi(t-r), \ \ \mbox{ for }\ t\geq r.
\label{psi-est}
\end{eqnarray}
Note that for each $t\geq r$ there exists a positive integer $j\geq 1$ such that $jr\leq t< (j+1)r$,
then by (\ref{df-psi}) and (\ref{psi-est}) we have
$$\varphi(t)\leq \psi(jr) \leq c_2r \psi(jr-r)\leq c_1(c_2r)^{j}.$$
Since $0<\rho\leq 1$, $0<c_2r<1$ and $r\leq jr\leq t< (j+1)r$,
then we get
\begin{eqnarray*}
\varphi(t)&\leq&  c_1(c_2r)^{j} \\
      &\leq& c_1 \exp(j\ln(c_2r))\\
          & \leq&   c_1 \exp((t/r-1)\ln(c_2r))\\
          & \leq& c_1 \exp(\rho(t/r-1)\ln(c_2r)).
\end{eqnarray*}
This, together with the facts that
$\exp(\rho(t/r-1)\ln(c_2r))\geq 1$ and $\varphi(t)\leq c_1$ for $0\leq t\leq r$,
yields that the second statement holds.
Therefore, the proof is now complete.
\end{proof}

Similar estimates as these in  Lemma \ref{lm-key-est} have been proved  by \cite{Arino-Pituk,Driver1976}.
However, here we adopt a different method to get a new estimate.
We also remark that to obtain the large spectral gap,
we actually need the constant $\rho$ to be in the interval $(0,1/2)$.

By the hypothesis {\bf (H)} and the inequality (1.4) in \cite[p.168]{JKHale-Verduyn},
we can obtain the estimates for the solutions of equation (\ref{NARFDE}).
More precisely,
we have the next lemma.
\begin{lemma} \label{bd-growth}
Assume that equation (\ref{NARFDE}) satisfies the hypothesis ${\bf (H)}$.
Then the solution $x(\cdot,t_{0},\phi)$ of equation (\ref{NARFDE}) with the initial value $x_{t_0}=\phi\in C[-r,0]$
satisfies the  following estimate
\begin{eqnarray*}
|x_{t}(\cdot,t_{0},\phi)|\leq e^{M(t-t_{0})}|\phi| \ \  \mbox{ for }\ \  t\geq t_{0}.
\end{eqnarray*}
\end{lemma}
\begin{proof}
This lemma can be proved by the similar way as in \cite[Theorem 1.1, p.\,168]{JKHale-Verduyn}.
Here we give a short proof.
Since the hypothesis ${\bf (H)}$ holds,
then
\begin{eqnarray*}
|L(t,\phi)|\leq M|\phi|, \ \ \ \ t\geq t_{0}, \ \ \phi\in C[-r,0].
\end{eqnarray*}
Thus by letting $h(s)=0$, $m(s)=M$ and $\sigma=t_{0}$ in  \cite[Formula (1.\,4), p.\,168]{JKHale-Verduyn},
we obtain the estimate in this lemma. This finishes the proof.
\end{proof}
By applying Theorem \ref{thm-prty-sp-solu}, and Lemmas \ref{lm-key-est} and \ref{bd-growth},
we can prove the following results, which are useful in giving the values of the constants $K_{2,r}$ and $\beta_{r}$.
\begin{lemma}\label{key-est2}
For each $\phi\in  C[-r,0] $,
let $x$ be the solution of equation (\ref{NARFDE}) with the initial value $x_{t_0}=\phi$ and
\begin{eqnarray*}
y(t)=\Phi(t_0,t)x(t) \ \ \ \ \ \mbox{ for } \ \ t\geq t_0,
\end{eqnarray*}
where the matrix-valued function $\Phi$ is defined as in Theorem \ref{thm-prty-sp-solu}.
Then there exists a vector $l(t_{0},\phi)\in \mathbb{R}^n$ such that
\begin{eqnarray}
\lim_{t\to+\infty}y(t)=l(t_{0},\phi).
\label{limt-y}
\end{eqnarray}
Furthermore, suppose that the function $\phi$ satisfies $|\phi|\leq 1$ and $l(t_{0},\phi)=0$.
Then the solution $x$ satisfies the following estimate
\begin{eqnarray}
|x(t)| \leq  -\frac{2e^2}{\rho(r \lambda_{r})^{\rho-1}\ln(r \lambda_{r})}\exp\left(\left(\frac{\rho\ln(r \lambda_{r})}{r}- \lambda_{r} \right)(t-t_0)\right),
\label{est-key-1}
\end{eqnarray}
for each $t\geq t_{0}$ and each $\rho\in (0,1]$.
\end{lemma}

\begin{proof}
By (iv) in Theorem \ref{thm-prty-sp-solu},
 we see that the function $y$ is well-defined.
To prove the existence of the limit,
we first give another property of the special matrix solution
$\Phi(\cdot,t_{0})=(x^{1}(\cdot,t_{0}),...,x^{n}(\cdot,t_{0}))$,
that is,
\begin{eqnarray}
\dot \Phi(t_0,t)=-\Phi(t_0,t)(L(t,x^1_t),L(t,x^2_t),...,L(t,x^n_t))\Phi(t_0,t),
\label{eq-Phi}
\end{eqnarray}
where the derivative is with respect to $t$.
For each $\xi\in \mathbb{R}^n$,
by (iv) in Theorem \ref{thm-prty-sp-solu} we see that $$\Phi(t_0,t)\Phi(t,t_0)\xi=\xi.$$
Then we obtain that
$$\dot \Phi(t_0,t)\Phi(t,t_0)\xi+\Phi(t_0,t)\dot\Phi(t,t_0)\xi=0,$$
which yields that
\begin{eqnarray*}
\dot \Phi(t_0,t)\Phi(t,t_0)\xi&=& -\Phi(t_0,t) L(t,\Phi_t\xi)\\
           &=& -\Phi(t_0,t)(L(t,x^1_t),L(t,x^2_t),...,L(t,x^n_t))\xi.
\end{eqnarray*}
By the arbitrariness of $\xi$, we get
$$
\dot \Phi(t_0,t)\Phi(t,t_0)= -\Phi(t_0,t)(L(t,x^1_t),L(t,x^2_t),...,L(t,x^n_t)).
$$
This yields that (\ref{eq-Phi}) holds.
Then for $t\geq t_0$,
\begin{eqnarray*}
\dot y(t)\!\!\!&=&\!\!\!\dot \Phi(t_0,t)x(t)+\Phi(t_0,t)\dot x(t)\\
\!\!\!&=&\!\!\! -\Phi(t_0,t)(L(t,x^1_t),L(t,x^2_t),...,L(t,x^n_t))\Phi(t_0,t)x(t)+\Phi(t_0,t)L(t,x_t)\\
\!\!\!&=&\!\!\! -\Phi(t_0,t)L(t,\Phi_t y(t))+\Phi(t_0,t)L(t, \Phi_t y_t),
\end{eqnarray*}
which yields
$$
\Phi(t,t_0)\dot y(t)=L(t,\Phi_t(y_t-y(t)))
$$
for $t\geq t_0$.
Then by (v) in Theorem \ref{thm-prty-sp-solu}, we get that for $t\geq t_0+r$,
\begin{eqnarray*}
|\Phi(t,t_0)\dot y(t)|&\leq&M \sup_{-r\leq \theta\leq 0}|\int_{t+\theta}^{t}\Phi(t+\theta,t_0)\dot y(\tau) d \tau|\\
    &=& M \sup_{-r\leq \theta\leq 0}|\int_{t+\theta}^{t}\Phi(t+\theta,\tau)\Phi(\tau,t_0)\dot y(\tau) d \tau|\\
    &\leq& M \sup_{-r\leq \theta\leq 0}|\int_{t+\theta}^{t} e^{ \lambda_{r} (\tau-t-\theta)}|\Phi(\tau,t_0)\dot y(\tau)| d \tau|\\
    &\leq&  M e^{r \lambda_{r}}\int_{t-r}^{t} e^{ \lambda_{r} (\tau-t)}|\Phi(\tau,t_0)\dot y(\tau)| d \tau.
\end{eqnarray*}
Since $M e^{r \lambda_{r}}=\lambda_{r}$ and $0<r\lambda_{r}<1$, then by Lemma \ref{lm-key-est}
we have that for  $t\geq t_{0}$,
\begin{eqnarray}
\begin{aligned}
e^{ \lambda_{r} (t-t_0)}|\Phi(t,t_0)\dot y(t)|
  &\leq   \lambda_{r} \int_{t-r}^{t} e^{ \lambda_{r} (\tau-t_0)}|\Phi(\tau,t_0)\dot y(\tau)| d \tau\\
  &\leq     \frac{\kappa_1}{(r \lambda_{r})^{\rho}}\exp\left(\frac{\rho\ln(r \lambda_{r})}{r}(t-t_0)\right),\\
\end{aligned}
\label{Phi-y-est}
\end{eqnarray}
where the constant $\kappa_1$ is in the form
$$\kappa_1=\sup_{t_0\leq t\leq t_0+r}e^{ \lambda_{r} (t-t_{0})}|\Phi(t,t_0)\dot y(t)|.$$
Hence, by Theorem \ref{thm-prty-sp-solu} and (\ref{Phi-y-est}) we obtain that for $t\geq t_{0}$,
\begin{eqnarray*}
|\dot y(t)|&\leq& |\Phi(t_0,t)||\Phi(t,t_0)\dot y(t)|\\
           &\leq& \frac{\kappa_1}{(r \lambda_{r})^{\rho}}\exp\left(\frac{\rho\ln(r \lambda_{r})}{r}(t-t_0)\right).
\end{eqnarray*}
This together with $r^{-1}\ln(r\lambda_{r})<0$ yields
that the limit of $y(t)$ exists as $t\to+\infty$.
Thus, the first statement is proved.

Suppose that $l(t_{0},\phi)=0$, then for $t\geq t_0$,
\begin{eqnarray}
\begin{aligned}
x(t)&=\Phi(t,t_0)y(t)\\
    & =\Phi(t,t_0)\int_{+\infty}^{t}\dot y(\tau) d\tau \\
    & =\int_{+\infty}^{t} \Phi(t,\tau)\Phi(\tau,t_0)\dot y(\tau) d\tau,
\end{aligned}
\label{eq-x-equal}
\end{eqnarray}
then by (\ref{Phi-y-est}), (\ref{eq-x-equal}) and (v) in Theorem \ref{thm-prty-sp-solu},
we obtain that
\begin{eqnarray}
|x(t)|
\leq -\frac{\kappa_1 r}{\rho(r \lambda_{r})^{\rho}\ln(r \lambda_{r})}\exp\left(\left(\frac{\rho\ln(r \lambda_{r})}{r}- \lambda_{r} \right)(t-t_0)\right),
 \ \ t\geq t_{0}.\label{xt-est}
\end{eqnarray}
For $|\phi|\leq1$ and $t_0\leq t\leq t_0+r$, we observe that
\begin{eqnarray}
\begin{aligned}
e^{ \lambda_{r} (t-t_0)}|\Phi(t,t_0)\dot y(t)|&= e^{ \lambda_{r} (t-t_0)}|L(t,\Phi_t(y_t-y(t)))|\\
         &\leq  Me^{r \lambda_{r}}\left(\sup_{-r\leq\theta\leq0}|\Phi_ty_t|+\sup_{-r\leq\theta\leq0}|\Phi_ty(t)|\right).
\end{aligned}
\label{phi-y-est00}
\end{eqnarray}
By Lemma \ref{bd-growth} we obtain
\begin{eqnarray}
\begin{aligned}
\sup_{0\leq t\leq r}\sup_{-r\leq\theta\leq0}|\Phi_ty_t| & =\sup_{-r\leq\theta\leq r}|x(t)|=\max\{1,\sup_{0\leq t\leq r}|x(t)|\}\\
   & \leq \max\{1,e^{Mr}\}=e^{Mr},
\end{aligned}
\label{phi-y-est01}
\end{eqnarray}
\begin{eqnarray}
\begin{aligned}
\sup_{0\leq t\leq r}\sup_{-r\leq\theta\leq0}|\Phi_ty(t)|
&=\sup_{0\leq t\leq r}\sup_{-r\leq\theta\leq0} |\Phi(t+\theta,t)\Phi(t,t_0)y(t)|\\
& \leq e^{r \lambda_{r}+Mr}.
\end{aligned}
\label{phi-y-est02}
\end{eqnarray}
Since $Me^{r \lambda_{r}}= \lambda_{r}$, $0<r \lambda_{r}<1$ and $Mer<1$,
then by (\ref{phi-y-est00}\,-\,\ref{phi-y-est02}) we obtain
\begin{eqnarray}
\kappa_1\leq  \lambda_{r}  e^{Mr}(1+e^{r \lambda_{r}})\leq  \lambda_{r}  e^{Mr}e^{r \lambda_{r}}(1+e^{-r \lambda_{r}})<2 \lambda_{r}  e^2.
\label{kappa-est}
\end{eqnarray}
Substituting (\ref{kappa-est}) into (\ref{xt-est}) yields the last statement.
Therefore, the proof is  complete.
\end{proof}

We remark that the results in Lemma \ref{key-est2}
lead to a significantly different presentation from  \cite[Theorem 4]{Driver1976},
from which we can only obtain the limit  in (\ref{limt-y}).
However,
to investigate the effects of small delay on  the spectral gap,
the angular distance and the separation index,
it is necessary to obtain the estimate (\ref{est-key-1}).

{\bf Proof of Theorem \ref{thm-main-results}.}
By \cite[Theorem 4.1]{Arino-Pituk},
for each $t\in\mathbb{R}$ the projection $P(t)$ is in the form $P(t)\phi=\Phi_{t}(\cdot,t)l(t,\phi)$ for each $\phi\in C[-r,0] $,
where the vector $l(t,\cdot)$ is defined as in Lemma \ref{key-est2}.
By Theorem  \ref{thm-prty-sp-solu}, we have that
for $t\geq s$ and $\phi\in C[-r,0] $,
\begin{eqnarray}
\begin{aligned}
|T(s,t)P(t)\phi|&=|T(s,t)\Phi_{t}(\cdot,t)l(t,\phi)|\\
               & =|\Phi_{s}(\cdot,t)l(t,\phi)|\\
         & =\sup_{-r\leq \theta\leq 0}|\Phi(s+\theta,t)l(t,\phi)|\\
          &   =\sup_{-r\leq \theta\leq 0}|\Phi(s+\theta,t+\theta)\Phi(t+\theta,t)l(t,\phi)|\\
         & \leq e^{\lambda_{r}|t-s|}|\Phi_{t}(\cdot,t)l(t,\phi)|\\
         &  =e^{\lambda_{r}|t-s|}|P(t)\phi|.
\end{aligned}
\label{est-unstab}
\end{eqnarray}
For each $s\in \mathbb{R}$ and each $\phi\in \mathcal{R}(I-P(s))$ with $|\phi|\leq1$,
let $x$ be the solution of equation (\ref{NARFDE}) with $x_{s}=\phi$.
Then $l(s,\phi)=0$ and by Lemma \ref{key-est2} we have
\begin{eqnarray}
|x_{t}|\leq  -\frac{2e^2}{\rho(r \lambda_{r})^{\rho-1}\ln(r \lambda_{r})}\exp\left(\left(\frac{\rho\ln(r \lambda_{r})}{r}- \lambda_{r} \right)(t-s-r)\right), \ \ \ t\geq s.
\label{est-stab}
\end{eqnarray}
Hence, (\ref{est-unstab}) and (\ref{est-stab}) yield  (\ref{constants}),
that is,
\begin{eqnarray*}
\begin{aligned}
&\alpha_{r}=- \lambda_{r} ,\ \ \
\beta_{r}=\frac{\rho\ln(r \lambda_{r})}{r}- \lambda_{r},\\
&K_{1,r}=1, \ \ \
K_{2,r}=-\frac{2e^{2+r \lambda_{r}}(r \lambda_{r})^{1-2\rho}}{\rho\ln(r \lambda_{r})}.
\end{aligned}
\end{eqnarray*}
Thus, {\rm (i)} is obtained.

By Lemma \ref{lm1.1},
we have that $r \lambda_{r}\to 0$ and $\lambda_{r}\to M$ as $r\to 0+$,
then  as $r\to 0+$,
\begin{eqnarray*}
\begin{aligned}
\alpha_{r}-\beta_{r} & =- \lambda_{r}-\left(\frac{\rho\ln(r \lambda_{r})}{r}- \lambda_{r}\right)\\
         &=-\frac{\rho\ln(r \lambda_{r})}{r}\\
         &=-\frac{\rho\lambda_{r}\ln(r \lambda_{r})}{r\lambda_{r}}
        \to +\infty.
\end{aligned}
\end{eqnarray*}
To prove (ii),
for each $s\in \mathbb{R}$ we denote $q_{+}:=|P(s)|$ and $q_{-}:=|I-P(s)|$.
Then from (\ref{ged-na-est-1}) and (\ref{ged-na-est-2}) it follows that
for each $\tau>0$,
\begin{eqnarray}
|T(s+\tau,s)P(s)|\geq (K_{1,r})^{-1} e^{\alpha_{r}\tau} q_{+},\\
|T(s+\tau,s)(I-P(s))|\leq K_{2,r}e^{\beta_{r}\tau}q_{-}.
\label{ineq-prj-1}
\end{eqnarray}
Let the function $\Psi$ be defined by
\begin{eqnarray*}
\Psi(\tau)=(K_{1,r})^{-1} e^{\alpha_{r}\tau}-K_{2,r}e^{\beta_{r}\tau}  \ \ \ \mbox{ for }\ \tau>0.
\end{eqnarray*}
Since $\alpha_{r}>\beta_{r}$,
then we can compute that
the function $\Psi$ satisfies that
$\Psi(\tau_{0})=0$ and $\Psi(\tau)>0$ for $\tau>\tau_{0}$,
where the constant  $\tau_{0}$ is in the form
\begin{eqnarray*}
\tau_{0}:=\frac{\ln(K_{1,r}K_{2,r})}{\alpha_{r}-\beta_{r}}>0.
\end{eqnarray*}
Applying Lemma \ref{bd-growth} and  (\ref{ineq-prj-1}) yields that for each $\tau>\tau_{0}>0$,
\begin{eqnarray}
\begin{aligned}
\Psi(\tau)&\leq||T(s+\tau,s)P(s)|q_{+}^{-1}-|T(s+\tau,s)(I-P(s))|q_{-}^{-1}|\\
          &\leq|T(s+\tau,s)P(s)q_{+}^{-1}+T(s+\tau,s)(I-P(s))q_{-}^{-1}|\\
          &\leq |T(s+\tau,s)||q_{+}^{-1}P(s)+q_{-}^{-1}(I-P(s))|\\
          &\leq e^{M\tau}|q_{+}^{-1}P(s)+q_{-}^{-1}(I-P(s))|\\
          &\leq e^{M\tau}|q_{-}^{-1}I+(q_{+}^{-1}-q_{-}^{-1})P(s)|\\
          &\leq e^{M\tau}(q_{-}^{-1}+|q_{+}^{-1}-q_{-}^{-1}|q_{+})\\
          &\leq e^{M\tau}q_{-}^{-1}(1+|q_{+}-q_{-}|) \\
          & \leq 2e^{M\tau}q_{-}^{-1},
\end{aligned}
\label{ineq-prj-2}
\end{eqnarray}
where we use the fact that
\begin{eqnarray*}
|q_{+}-q_{-}|=||P(s)|-|I-P(s)||\leq |P(s)+(I-P(s))|=1.
\end{eqnarray*}
Recall that  the constant $\gamma_{r}$ is in the form
$$\gamma_{r}=\frac{M-\beta_{r}}{\alpha_{r}-\beta_{r}}.$$
Since $M>\alpha_{r}$ and $\alpha_{r}>\beta_{r}$,
then we can check that  $\gamma_{r}>1$.
Let the function $\widetilde{\Psi}$ be defined by
$$\widetilde{\Psi}(\tau):=2e^{M\tau} (\Psi(\tau))^{-1}, \ \ \ \ \  \tau>\tau_{0}.$$
Then we can compute that the function $\widetilde{\Psi}$
reaches the minimum value in $(\tau_{0},+\infty)$ at
\begin{eqnarray*}
\tau_{1}:=\frac{\ln(K_{1,r}K_{2,r}(M-\beta_{r}))-\ln(M-\alpha_{r})}{\alpha_{r}-\beta_{r}}.
\end{eqnarray*}
Substituting $\tau=\tau_{1}$ into $\widetilde{\Psi}(\tau)$ yields that
\begin{eqnarray*}
\min_{\tau>\tau_{0}}\widetilde{\Psi}(\tau)&=&\widetilde{\Psi}(\tau_{1})\\
&=&
\frac{2K_{1,r}}{\left(\frac{K_{1,r}K_{2,r}(M-\beta_{r})}{M-\alpha_{r}}\right)^{1-\gamma_{r}}
-K_{1,r}K_{2,r}\left(\frac{K_{1,r}K_{2,r}(M-\beta_{r})}{M-\alpha_{r}}\right)^{-\gamma_{r}}}\\
&=& 2\gamma_{r} K_{1,r}\left(\frac{K_{1,r}K_{2,r}(M-\beta_{r})}{M-\alpha_{r}}\right)^{\gamma_{r}-1}
\\
&=& 2\gamma_{r} K_{1,r}\left(1+\frac{1}{\gamma_{r}-1}\right)^{\gamma_{r}-1}(K_{1,r}K_{2,r})^{\gamma_{r}-1},
\end{eqnarray*}
where we use the fact that
\begin{eqnarray*}
1+\frac{1}{\gamma_{r}-1}=\frac{M-\beta_{r}}{M-\alpha_{r}}.
\end{eqnarray*}
Thus, by (\ref{ineq-prj-2}) we have
\begin{eqnarray}
q_{-}\leq 2\gamma_{r} K_{1,r}\left(1+\frac{1}{\gamma_{r}-1}\right)^{\gamma_{r}-1}(K_{1,r}K_{2,r})^{\gamma_{r}-1}.
\label{est-prj}
\end{eqnarray}
Since $K_{r}=\max\{K_{1,r},\,K_{2,r}\}$, $\gamma_{r}>1$ and
\begin{eqnarray*}
\left(1+\frac{1}{\gamma_{r}-1}\right)^{\gamma_{r}-1}<e,
\end{eqnarray*}
then by (\ref{est-prj}) we have that
$$q_{-}\leq 2e\gamma_{r}(K_{r})^{2\gamma_{r}-1}.$$
Similarly, the above inequalities also hold for $q_{+}$.
Thus, {\rm (ii)} is proved.

By Lemma \ref{lm1.1}, for each $\rho\in (0,1/2)$ we have that
as $r\to 0+$, the following limits hold:
\begin{eqnarray}
\begin{aligned}
\beta_{r}&=\frac{\rho\ln(r \lambda_{r})}{r}- \lambda_{r}\to -\infty, \\
  \gamma_{r}&=\frac{M-\beta_{r}}{\alpha_{r}-\beta_{r}}\to 1, \\
\frac{\alpha_{r}-\beta_{r}}{(K_{1,r})^{2\gamma_{r}+1}}
      & =\alpha_{r}-\beta_{r}\to +\infty,\\
\frac{\alpha_{r}-\beta_{r}}{(K_{2,r})^{2\gamma_{r}+1}}
     &=\frac{ \lambda_{r} \rho^{2(\gamma_{r}+1)}}{(2\exp(2+r \lambda_{r}))^{2\gamma_{r}+1}}
     \frac{\left(\ln(r \lambda_{r})\right)^{2(\gamma_{r}+1)}}{(r \lambda_{r})^{(1-2\rho)(2\gamma_{r}+1)+1}}
    \to +\infty.
\end{aligned}
\label{lmt-1-2}
\end{eqnarray}
By (\ref{est-prj}) and the definition of the constant $\mathcal{L}_{r}$ given in (\ref{gap-cond}),
we obtain
\begin{eqnarray*}
\mathcal{L}_{r}\geq \frac{\alpha_{r}-\beta_{r}}{8 e \gamma_{r}(K_{r})^{2\gamma_{r}+1}},
\end{eqnarray*}
then by (\ref{lmt-1-2}) the limit in  (\ref{gap-cond}) holds.
Thus, {\rm (iii)} is proved.
This finishes the proof.
\hfill $\Box$

To prove the last theorem,
we introduce several results on the angular distance and  the separation index,
which show the relationships between these geometric properties of a splitting and the corresponding projections.
These results are summarized in the following lemma.

\begin{lemma}
Let the Banach space $C[-r,0]$ have a splitting $C[-r,0]=\mathcal{C}_{1}\oplus \mathcal{C}_{1}$,
where $\mathcal{C}_{j}$ are nonzero closed subspaces of $C[-r,0]$,
and $P_{1}$ and $P_{2}=I-P_{1}$ are the corresponding projections with $\mathcal{R}(P_{1})=\mathcal{C}_{1}$
and $\mathcal{R}(P_{2})=\mathcal{C}_{2}$. Then the following statements hold:
\begin{itemize}
\item[(i)]
\label{state-1}
the angular distance $\gamma(\mathcal{C}_{1},\mathcal{C}_{2})$ defined as in (\ref{df-distance}) satisfies the following estimates:
\begin{eqnarray}
\label{est-imp-1}
(\gamma(\mathcal{C}_{1},\mathcal{C}_{2}))^{-1}\leq |P_{j}| \leq 2(\gamma(\mathcal{C}_{1},\mathcal{C}_{2}))^{-1}.
\end{eqnarray}

\item[(ii)]
\label{state-2}
the separation indices $\underline{\rm dist}(\mathcal{C}_{1},\mathcal{C}_{2})$ and
$\underline{\rm dist}(\mathcal{C}_{2},\mathcal{C}_{1})$ defined as in (\ref{df-index})  satisfy that
\begin{eqnarray}
\label{est-imp-2}
\underline{\rm dist}(\mathcal{C}_{1},\mathcal{C}_{2})=(|P_{1}|)^{-1}, \ \ \ \
\underline{\rm dist}(\mathcal{C}_{2},\mathcal{C}_{1})=(|P_{2}|)^{-1}.
\end{eqnarray}
\end{itemize}
\label{lm-dist-index}
\end{lemma}
The statement (\ref{est-imp-1}) is obtained by \cite[Lemma 1.1, p.156]{Daleckii-Krein-74},
and the statement (\ref{est-imp-2}) is given in \cite[Lemma 3]{Lian-Wang-15}.
Thus the detailed proof is omitted.

By applying these statements in Lemma \ref{lm-dist-index} and Theorem \ref{thm-angular-index},
we can given the proof for the final theorem.

{\bf Proof of Theorem \ref{thm-angular-index}.}
For each $s\in \mathbb{R}$,
let the subspaces $\mathcal{C}_{\pm}$ be in the form
\begin{eqnarray*}
\mathcal{C}_{+}=\mathcal{R}(P(s)), \ \ \ \ \  \mathcal{C}_{-}=\mathcal{R}(I-P(s)),
\end{eqnarray*}
where $P(s)$ is the projection operator defined as in Theorem \ref{thm-ed}.
Take $\rho$ with $\rho\in (0,1/2)$ in (\ref{constants}).
Then the fact that $r\lambda_{r}\to 0$ as $r\to 0+$, which is obtained in Lemma \ref{lm1.1},
yields that
\begin{eqnarray*}
K_{2,r}=-\frac{2e^{2+r \lambda_{r}}(r \lambda_{r})^{1-2\rho}}{\rho\ln(r \lambda_{r})}\to 0, \ \ \ \mbox{ as }\ r\to 0+.
\end{eqnarray*}
This implies that $K_{r}=\max\{K_{1,r},\,K_{2,r}\}=1$ for sufficiently small $r$.
As stated in (\ref{lmt-1-2}),
we have the limit:
\begin{eqnarray*}
\gamma_{r}=\frac{M-\beta_{r}}{\alpha_{r}-\beta_{r}}\to 1, \ \ \ \mbox{ as }\ r\to 0+.
\end{eqnarray*}
Then the right side of (\ref{est-prj-01}) satisfies that
\begin{eqnarray*}
2e\gamma_{r}(K_{r})^{2\gamma_{r}-1}\to 2e,  \ \ \ \mbox{ as }\ r\to 0+.
\end{eqnarray*}
Thus by continuity
there exists a sufficiently small $\widetilde{r}_{0}$ with $0<\widetilde{r}_{0}<1/(Me)$
and a positive constant $\delta$ that is independent of $r$ in $(0,\widetilde{r}_{0}]$ such that
\begin{eqnarray}
2e\gamma_{r}(K_{r})^{2\gamma_{r}-1}\leq \delta^{-1}.
\label{est-1}
\end{eqnarray}
Applying (\ref{est-prj-01}), (\ref{est-1}) and (\ref{est-imp-2}) in Lemma \ref{lm-dist-index} ,
we obtain that
\begin{eqnarray}
\underline{\rm dist}(\mathcal{C}_{+},\mathcal{C}_{-})\geq \delta>0
\label{est-2}
\end{eqnarray}
for each $r\in(0,\widetilde{r}_{0}]$.
By (\ref{est-imp-1}) and (\ref{est-imp-2}) in Lemma \ref{lm-dist-index},
we obtain that $\gamma(\mathcal{C}_{+},\mathcal{C}_{-})$ and $\underline{\rm dist}(\mathcal{C}_{+},\mathcal{C}_{-})$ satisfy
\begin{eqnarray*}
\gamma(\mathcal{C}_{+},\mathcal{C}_{-}) \geq \underline{\rm dist}(\mathcal{C}_{+},\mathcal{C}_{-}).
\end{eqnarray*}
This together with (\ref{est-2}) yields
$$0<\delta\leq \underline{\rm dist}(\mathcal{C}_{+},\mathcal{C}_{-}) \leq \gamma(\mathcal{C}_+,\mathcal{C}_-).$$
Similarly, we can obtain
$$0<\delta\leq \underline{\rm dist}(\mathcal{C}_{-},\mathcal{C}_{+}) \leq \gamma(\mathcal{C}_+,\mathcal{C}_-).$$
Therefore, the proof is now complete.
\hfill $\Box$

\begin{remark}
As the delay tends to zero,
we prove that the spectral gap corresponding to this dichotomy can be large enough,
and the angular distance  and the separation index are uniformly bounded from below.
We hope that these results stated in this paper will be used to study Hopf bifurcation,
canard explosion and relaxation oscillations
arising from differential equations with small delays,
and we refer the readers to the recent works \cite{Campbel-Stone-Erneux,Erneux-etal,Krupa-Touboul,Kuehn-15}.

We also point out that compared to linear autonomous retarded differential equations,
the semigroups generalized by the solution operators of neutral differential equations have not only point spectrum
but also continuous spectrum \cite{Hale-Lunel2001,Henry1974}.
This causes a big obstacle in establishing the existence of pseudo-exponential dichotomies.
Recently,
large spectral gaps induced by small delays for neutral differential equations were considered in \cite{Chen-Shen-20},
where the proof is finished by the {\it Banach Fixed Point Theorem}.
Here we apply a more straightforward way to obtain the same result for retarded differential equations and
further study the effects of small delays on the angular distance and the separation index.
\end{remark}

\section*{References}





\end{document}